\documentclass[12pt,twoside,english,a4paper]{article}
\usepackage{babel,amssymb,amsthm}
\usepackage{graphicx}
\usepackage{amsmath}
\usepackage{bm}
\usepackage{color}
\usepackage{framed}
\usepackage{hyperref}

\newcommand{\R}{\mathbb R}
\newcommand{\C}{\mathbb C}
\newcommand{\wt}[1]{\widetilde{#1}}

\newcommand{\jump}[1]{[\![#1]\!]}
\newcommand{\ave}[1]{\{\!\{#1\}\!\}}

\newcommand{\ntr}{\gamma_\nu}
\newcommand{\nd}{\partial_\nu}
\newcommand{\Gscr}{{\Gamma_{\mathrm{scr}}}}

\newcommand{\bff}[1]{{\mathbf{#1}}}

\renewcommand{\SS}{{\mathrm S}}
\newcommand{\DD}{{\mathrm D}}
\newcommand{\VV}{{\mathrm V}}
\newcommand{\WW}{{\mathrm W}}
\newcommand{\KK}{{\mathrm K}}
\newcommand{\II}{{\mathrm I}}

\newcommand{\CS}{{\mathcal S}}
\newcommand{\CD}{{\mathcal D}}
\newcommand{\CV}{{\mathcal V}}
\newcommand{\CW}{{\mathcal W}}
\newcommand{\CK}{{\mathcal K}}

\newcommand{\Rd}{{\R^d}}
\newcommand{\RdG}{{\R^d\setminus\Gamma}}
\newcommand{\divv}{{\mathrm{div}}}
\newcommand{\TD}{\mathrm{TD}}
\newcommand{\Hdiv}[1]{\mathbf H(\divv,#1)}
\newcommand{\Hhalf}{H^{1/2}(\Gamma)}
\newcommand{\Hmhalf}{H^{-1/2}(\Gamma)}

\newcommand{\Vh}{{\mathbf V_h}}

\definecolor{red}{rgb}{0.8,0,0} 

\newtheorem{proposition}{Proposition}[section]

\newtheorem{theorem}[proposition]{Theorem}

\numberwithin{equation}{section}

\title{A new and improved analysis of the time domain boundary integral operators  for acoustics}

\author{
Matthew Hassell,\footnote{E-mail: \href{mailto:mhassell@udel.edu}{\texttt{mhassell@udel.edu}}}
\and
Tianyu Qiu,\footnote{E-mail: \href{mailto:qty@udel.edu}{\texttt{qty@udel.edu}}}
\and
Tonatiuh S\'anchez-Vizuet,\footnote{E-mail: \href{mailto:tonatiuh@udel.edu}{\texttt{tonatiuh@udel.edu}}}
\and
Francisco--Javier Sayas\footnote{E-mail: \href{mailto:fjsayas@udel.edu}{\texttt{fjsayas@udel.edu}}, Partially funded by NSF (grant DMS 1216356).}
\and Department of Mathematical Sciences,
University of Delaware
}
\date{\today}

\begin{document}

\maketitle

\begin{abstract} We present a novel analysis of the boundary integral operators associated to the wave equation. The analysis is done entirely in the time-domain by employing tools from abstract evolution equations in Hilbert spaces and semi-group theory. We prove a single general theorem from which well-posedness and regularity of the solutions for several boundary integral formulations can be deduced as particular cases.  By careful choices of continuous and discrete spaces, we are able to provide a concise analysis for various direct and indirect formulations, both at the continuous level and for their Galerkin-in-space semi-discretizations.  Some of the results here are improvements on previously known results, while other results are equivalent to those in the literature.  The methodology presented here greatly simplifies the analysis of the operators of the Calder{\'o}n projector for the wave equation and can be generalized for other relevant boundary integral equations.  
\end{abstract}

\section{The context and the goals}

We present a new technique for direct-in-time analysis of the operators of the Calder{\'o}n projector for the acoustic wave equation. The analysis is carried out by first formulating the wave equation as a first-order-in-time-and-in-space transmission problem. We then show that this exotic transmission problem generates a strongly continuous group ($C_0$ group) of isometries in an appropriately chosen Hilbert space.  From this abstract dynamical system, we are able to derive stability and error estimates for a variety of transient scattering problems, both continuous and semi-discrete in space. This new technique offers a number of improvements over the Laplace domain analysis that originated in \cite{BaHa:1986a} and \cite{BaHa:1986b}, which carries out inversion using a Plancherel formula in anisotropic Sobolev spaces. Later work in \cite{Lubich:1994} used the Laplace domain method and derived time-domain estimates by inversion of the Laplace transform. The Laplace domain analysis was given a systematic treatment in \cite{LaSa:2009a} for acoustic waves, and has been applied to numerous other problems, such as electromagnetic scattering \cite{BaBaSaVe:2013}, electromagnetic transmission \cite{ChMo:2015}, and wave-structure interaction \cite{HsSaSa:2015}. A detailed outline of the Laplace domain analysis of transient acoustic scattering can be found in the first part of \cite{Sayas:2014}.    

The direct-in-time study of the acoustic Calder{\'o}n projector began in \cite{Sayas:2013d} and was detailed in the second part of \cite{Sayas:2014}, employing a second order (in time and in space) equation approach, namely, the problems were rewritten as a second-order-in-time differential equation associated to an unbounded (second order differential) operator in the space variables. This approach later proved to be inflexible for the treatment of Maxwell equations, this lead to the use of semigroup theory in \cite{QiSa:2015}, greatly simplifying the analysis and sidestepped the cut-off process and reconciliation step described in \cite{Sayas:2014, BaLaSa:2015, QiSa:2014}. Moreover, the estimates obtained with the direct-in-time analysis are sharper than those obtained through Laplace domain analysis. In particular, the dependence on time is made explicit and the temporal regularity for the input data is lowered. We will remark on such improvements in the course of this article.

We present here a single theorem that covers all of the possible problems of interest as special cases. By choosing the appropriate spaces we are able to systematically derive estimates for the time domain layer potentials, time domain boundary integral operators, time domain DtN/NtD maps, semi-discrete Galerkin solver and error operators for direct/indirect/symmetric formulations for interior/exterior Dirichlet/Neumann equations. The theory also covers screens and mixed boundary conditions without any modification.  We are hopeful that this is the final ``big theorem" which unites all of the previously developed direct-in-time analysis and that a more general framework will not be needed in the future.

The paper slowly builds up the required material in order to state and prove the abstract theorem and then proceeds to apply it for specific cases. It is organized as follows. In Section \ref{sec:2}, we introduce the background material on Sobolev spaces, the potentials and operators for the acoustic wave equation, and their mapping properties. Section  \ref{sec:3} builds the key theorems on an abstract evolution equations on a Hilbert space from which all of the main results will follow. Section \ref{sec:4} applies the previous result to a particular dynamical system that arises from our study of the acoustic wave equation in an abstract setting. We then formulate the various integral representations as a single exotic transmission problem from which all of the specific formulations follow via careful choices of spaces and data. Section \ref{sec:5} is a summary of the estimates that follow from the theorems in Section \ref{sec:4}. We conclude by pointing at some possible extensions.

\paragraph{Background.} Section \ref{sec:2} gives a fast introduction to the PDE and distribution theory background needed for this paper. All ideas on Sobolev spaces and steady-state layer potentials on Lipschitz domains can be found in McLean's monograph on elliptic systems \cite{McLean:2000}. For necessary results on vector-valued distributions and their Laplace transforms we refer to the Dautray-Lions encyclopedia \cite{DaLi:1992}. A compendium of what is needed in the context of time domain integral equations can be found in \cite{Sayas:2014}. Finally, some basic results on semigroups of operators will be used: the most elementary ones can be easily found in functional analysis textbooks \cite{Kesavan:1989}, while the results on the behaviour of nonhomogeneous can be found in Pazy's well known monograph \cite{Pazy:1983}.

\section{The materials}\label{sec:2}

This paper is a compendium of new and old techniques that build on a relatively vast body of knowledge. This section is devoted to introducing all the necessary tools to present the time domain integral operators for the wave equation.

The geometric setting of this paper is as follows. The open set $\Omega_-\subset \Rd$ is the union of a finite collection of bounded open sets $\Omega_i$ ($i=1,\ldots,N$) with connected Lipschitz boundaries. We assume that the closures of the components $\Omega_i$ do not intersect. We write $\Gamma:=\partial \Omega_-=\cup_{i=1}^N \partial\Omega_i$ and $\Omega_+:=\Rd\setminus\overline{\Omega_-}$.

\paragraph{Sobolev space notation. }
Given an open set $\mathcal O$ (in this paper $\mathcal O\in \{\Rd,\RdG,\Omega_+,\Omega_-\}$), we denote
\[
(u,v)_{\mathcal O}:=\int_{\mathcal O} u\,v, \qquad (\bff u,\bff v)_{\mathcal O}:=\int_{\mathcal O}\bff u\cdot\bff v.
\]
This is the inner product of $L^2(\mathcal O)$ and $\bff L^2(\mathcal O)$ in the real case. In the complex case, the bracket will still be linear and we will need to conjugate to get the inner product. We also denote
\[
\| u\|_{\mathcal O}:=\sqrt{(u,\overline{u})_{\mathcal O}},
\qquad
\|\bff u\|_{\mathcal O}:=\sqrt{(\bff u,\overline{\bff u})_{\mathcal O}}.
\]
The space $H^1(\mathcal O)$ is the standard Sobolev space and $\Hdiv{\mathcal O}:=\{ \bff v\in \bff L^2(\mathcal O)\,:\, \nabla\cdot\bff v\in L^2(\mathcal O)\}$.
The $H^1(\mathcal O)$ norm is denoted $\|\cdot\|_{1,\mathcal O}$ and the $\Hdiv{\mathcal O}$ norm is denoted $\|\cdot\|_{\divv,\mathcal O}$. For Lipschitz boundaries we consider the trace space $\Hhalf$ and denote by $\Hmhalf$ the  representation of its dual space obtained when the dual of $L^2(\Gamma)$ is identified with itself. The duality product $\Hmhalf\times \Hhalf$ will be denoted with angled brackets $\langle\cdot,\cdot\rangle_\Gamma$, linear in both components.

\paragraph{Traces.}
The following trace operators
\[
\gamma^\pm : H^1(\RdG) \to \Hhalf
\qquad
\gamma : H^1(\Rd)\to \Hhalf,
\]
are bounded and surjective. Given $u\in H^1(\RdG)$ we will denote
\[
\jump{\gamma u}:=\gamma^- u-\gamma^+ u,
\qquad
\ave{\gamma u}:=\tfrac12(\gamma^-u+\gamma^+u).
\]
The normal components for $\bff v\in \Hdiv{\Omega_\pm}$ are elements $\ntr^\pm \bff v\in \Hmhalf$ satisfying
\begin{eqnarray*}
\langle\ntr^-\bff v,\gamma^-w\rangle_\Gamma =(\nabla\cdot\bff v,w)_{\Omega_-}+(\bff v,\nabla w)_{\Omega_-}
& &  \forall w\in H^1(\Omega_-),\\
\langle\ntr^+\bff v,\gamma^+w\rangle_\Gamma =-(\nabla\cdot\bff v,w)_{\Omega_+}-(\bff v,\nabla w)_{\Omega_+}
& &  \forall w\in H^1(\Omega_+).
\end{eqnarray*}
We recall that $\ntr^\pm:\Hdiv{\Omega_\pm}\to \Hmhalf$ are surjective.
For $\bff v\in \Hdiv{\RdG}$ we can define
\[
\jump{\ntr\bff v}:=\ntr^-\bff v-\ntr^+\bff v,
\qquad
\ave{\ntr\bff v}:=\tfrac12 (\ntr^-\bff v+\ntr^+\bff v).
\]
When $\bff v\in \Hdiv{\Rd}$, we will write $\ntr\bff v=\ntr^\pm\bff v$. Finally, in the space
\[
H^1_\Delta(\Omega_\pm):=\{ u\in H^1(\Omega_\pm)\,:\,\nabla u\in \Hdiv{\Omega_\pm}\}
						=\{ u\in H^1(\Omega_\pm)\,:\, \Delta u\in L^2(\Omega_\pm)\},
\]
we can define $\nd^\pm u=\ntr^\pm\nabla u$. For $u\in H^1_\Delta(\RdG)$, we denote
\[
\jump{\nd u}:=\nd^-u-\nd^+u=\jump{\ntr\nabla u},\qquad
\ave{\nd u}:=\tfrac12(\nd^-u+\nd^+u)=\ave{\ntr\nabla u}.
\]
When $u\in H^1(\RdG)$ but $\nabla u\in \Hdiv{\Rd}$ we will write $\nd u=\nd^\pm u$.

\paragraph{Two remarks.} We will deal with evolution equations taking values in real Sobolev spaces. The complexifications of these spaces will appear when we take Laplace transforms. While Lebesgue integration over $\Rd$ and $\RdG$ is clearly the same, we will distinguish one set from the other when there is a differential operator in the integrand. For instance, $\|\nabla u\|_\RdG$ will be used for $u\in H^1(\RdG)$ and $\|\nabla u\|_{\Rd}$ will be used for $u\in H^1(\Rd)$. {\em Unless explicitly stated, all differential operators in the space variables, and the associated differential equations, will be assumed to be used in $\RdG$,}

\paragraph{Vector-valued distributions. }
Let $\mathcal D(\R)$ be the space of infinitely differentiable functions with compact support, endowed with its usual concept of convergence \cite{Schwartz:1966}.
Given a Banach space $X$, an $X$-valued distribution is a sequentially continuous linear map $f:\mathcal D(\mathbb R)\to X$, with the action of $f$ on $v\in \mathcal D(\mathbb R)$ denoted $\langle f,v\rangle_{\mathcal D'\times \mathcal D}$. A distribution is said to be causal when $\langle f,v\rangle_{\mathcal D'\times \mathcal D}=0$ whenever $\mathrm{supp}\, v\subset (-\infty,0)$. The derivative of a distribution $f$ is the distribution $\dot f$ given by $\langle \dot f,v\rangle_{\mathcal D'\times \mathcal D}=-\langle f,\dot v\rangle_{\mathcal D'\times \mathcal D}$. 

\begin{theorem}\label{th:FJS}
{\rm \cite[Chapter 3]{Sayas:2014}} Let $X$ be a Banach space and let $f$ be an $X$-valued distribution. The following statement on $f$
\begin{quote}
there exists a continuous function $g:\mathbb R\to X$ such that $g(t)=0$ for all $t\le 0$ and such that $\| g(t)\| \le C t^m$ for all $t\ge 1$ with $m\ge 0$, and there exists a non-negative integer $k$ such that $f=g^{(k)}$
\end{quote}
is equivalent to
\begin{quote}
$f$ admits a Laplace transform $\mathrm F=\mathcal L\{ f\}$ defined in $\mathbb C_+:=\{ s\in \mathbb C\,:\, \mathrm{Re}\,s>0\}$ and satisfying $\| \mathrm F(s)\|\le C_{\mathrm F}(\mathrm{Re}\,s) |s|^\mu$ for all $s\in \mathbb C_+$, where $\mu \in \mathbb R$ and $C_{\mathrm F}:(0,\infty)\to (0,\infty)$ is non-increasing and such that $C_{\mathrm F}(\sigma)\le C \sigma^{-\ell}$ for all $\sigma<1$ for some $C>0$ and $\ell \ge 0$.
\end{quote}
\end{theorem}

\paragraph{The TD class.}
Following \cite{Sayas:2014}, the set of all causal distributions characterized by Theorem \ref{th:FJS} will be denoted $\TD(X)$ (TD as in time-domain). Note that if $X$ and $Y$ are Hilbert spaces, $f\in \TD(X)$ and $A\in \mathcal B(X,Y)$, then $Af\in \TD(Y)$. In particular, if $X\subset Y$ with continuous embedding, $f\in \TD(X)$ implies that $f\in \TD(Y)$. When $f\in \TD(X)$, we will define $\partial^{-1}f\in \TD(X)$ by the equality $\mathcal L\{\partial^{-1} f\}(s)=s^{-1}\mathrm F(s)$. The operator $\partial^{-1}$ is a weak form of the causal antidifferentiation operator
\[
(\partial^{-1} f)(t)=\int_0^t f(\tau)\mathrm d\tau.
\]
For $f\in \mathcal C(\R_+;X)$, we define 
\[
Ef (t):=\left\{ \begin{array}{ll} f(t), & t\ge 0,\\
							0, & t< 0.
		\end{array}\right.
\]
If $\| f(t)\|\le C t^m$ for $t\ge 1$ and some non-negative integer $m$, then $Ef\in \TD(X)$. Also, if $f\in \mathcal C^1(\R_+;X)$ and $f(0)=0$, then
\[
\frac{\mathrm d}{\mathrm dt}(Ef)=E\dot f,
\]
where the derivative in the left-hand-side is in the sense of $X$-valued distributions, while the derivative in the right-hand-side is a classical derivative.
We will use the spaces
\[
\mathcal C^k_+(\R;X):=\{  f\in \mathcal C^k(\R;X)\,:\, f(t)=0 \quad t\le 0\},
\]
and
\[
W^k_+(\R;X):=\{ f\in \mathcal C^{k-1}_+(\R;X)\,:\, f^{(k)}\in L^1(\R;X)\}.
\]
Note that $W^k_+(\R;X)\subset \TD(X)$.

\paragraph{Laplace domain form of potentials and operators.}
For $s\in \C_+$, $\varphi\in \Hhalf$, $\lambda \in \Hmhalf$, the problem
\begin{alignat*}{6}
u\in H^1(\RdG)  \qquad & \Delta u-s^2 u=0 \qquad \mbox{in $\RdG$},\\
				 & \jump{\gamma u}=\varphi,\\
				 & \jump{\nd u}=\lambda,
\end{alignat*}
admits a unique solution. The variational formulation of this problem is
\begin{alignat*}{6}
u\in H^1(\RdG) \qquad & \jump{\gamma u}=\varphi,\\
				 & (\nabla u,\nabla v)_{\RdG}+s^2(u,v)_{\Rd}=\langle\lambda,\gamma v\rangle_\Gamma
						\quad \forall v\in H^1(\Rd).
\end{alignat*}
Its solution is denoted using two bounded linear operators
$u=\SS(s)\lambda-\DD(s)\varphi.$
By definition,
\[
\jump{\gamma}\SS(s)=0, 
	\qquad 
\jump{\nd}\SS(s) = \II,
	\qquad
\jump{\gamma}\DD(s)=-\II,
	\qquad
\jump{\nd}\DD(s)=0.
\]
We then define the four boundary integral operators
\begin{alignat*}{6}
& \VV(s)=\ave{\gamma}\SS(s)=\gamma^\pm \SS(s),
	&\qquad
		& \KK(s)=\ave{\gamma}\DD(s),\\
& \KK^t(s)=\ave{\nd}\SS(s),
	&\qquad
		&\WW(s)=-\ave{\nd}\DD(s)=-\nd^\pm \DD(s),
\end{alignat*}
and we have the limit relations
\[
\nd^\pm \SS(s)=\mp\tfrac12\II+\KK^t(s),
	\qquad
\gamma^\pm \DD(s)=\pm\tfrac12\II+\KK(s).
\]
The operators $\VV(s)$ and $\WW(s)$ are invertible. We will denote
\[
\VV^{-1}(s):=(\VV(s))^{-1}, 
	\qquad
\WW^{-1}(s):=(\WW(s))^{-1}.
\]

\begin{theorem}\label{th:bounds}
The following bounds hold for all $s\in \C_+$
\begin{alignat*}{6}
\| \SS(s)\|_{\Hmhalf\to H^1(\Rd)} &\le C \frac{|s|}{\sigma\underline\sigma^2}, \\
\| \DD(s)\|_{\Hhalf \to H^1(\RdG)} &\le C\frac{|s|^{3/2}}{\sigma\underline\sigma^{3/2}},\\
\| \VV(s)\|_{\Hmhalf\to\Hhalf}+\| \WW^{-1}(s)\|_{\Hmhalf\to\Hhalf}  & \le C \frac{|s|}{\sigma\underline\sigma^2},\\
\| \KK(s)\|_{\Hhalf\to\Hhalf}+\| \KK^t(s)\|_{\Hmhalf\to\Hmhalf}  & \le C\frac{|s|^{3/2}}{\sigma\underline\sigma^{3/2}},\\
\| \WW(s)\|_{\Hhalf\to\Hmhalf} + \| \VV^{-1}(s)\|_{\Hhalf\to\Hmhalf} & \le C\frac{|s|^2}{\sigma\underline\sigma}.
\end{alignat*}
In all of them we have denoted $\sigma:=\mathrm{Re}\,s$ and $\underline\sigma:=\min\{1,\sigma\}$.
\end{theorem}

\paragraph{Retarded potentials and operators.}
By Theorems \ref{th:FJS} and \ref{th:bounds} we can take the inverse Laplace transform of the operators and potentials defined above:
\begin{alignat*}{6}
\CS:=\mathcal L^{-1}\{\SS\} & \in \TD(\mathcal B(\Hmhalf,H^1_\Delta(\RdG))),\\
\CD:=\mathcal L^{-1}\{\DD\} & \in \TD(\mathcal B(\Hhalf,H^1_\Delta(\RdG))),\\
\CV:=\mathcal L^{-1}\{\VV\} &\in \TD(\mathcal B(\Hmhalf,\Hhalf)),\\
\CK:=\mathcal L^{-1}\{\KK\} &\in \TD(\mathcal B(\Hhalf,\Hhalf)),\\
\CK^t:=\mathcal L^{-1}\{\KK^t\} &\in \TD(\mathcal B(\Hmhalf,\Hmhalf)),\\
\CW:=\mathcal L^{-1}\{\WW\} &\in \TD(\mathcal B(\Hhalf,\Hmhalf)),\\
\CV^{-1}:=\mathcal L^{-1}\{\VV^{-1}\} &\in \TD(\mathcal B(\Hhalf,\Hmhalf)),\\
\CW^{-1}:=\mathcal L^{-1}\{\WW^{-1}\} &\in \TD(\mathcal B(\Hmhalf,\Hhalf)).
\end{alignat*}
The distributional version of Kirchhoff's formula can be stated by solving a transmission problem: given $\varphi\in \TD(\Hhalf)$ and $\lambda\in \TD(\Hmhalf)$ the unique solution to the problem
\begin{alignat*}{6}
u\in \TD(H^1_\Delta(\RdG)) \qquad
	& \ddot u=\Delta u,\\
	& \jump{\gamma u}=\varphi,\\
	& \jump{\nd u}=\lambda,
\end{alignat*}
is  $u=\CS*\lambda-\CD*\varphi.$ If we define $u=\CS*\lambda-\CD*\varphi,$ then
\[
\gamma^\pm u=\CV*\lambda-\CK*\varphi\mp\tfrac12\varphi,
	\qquad
\nd^\pm u=\mp\tfrac12\lambda+\CK^t*\lambda+\CW*\varphi.
\]

\section{The framework}\label{sec:3}

\paragraph{Function spaces and operators.} Let $\mathbb H$, $\mathbb V$,  $\mathbb M_1$ and $\mathbb M_2$ be Hilbert spaces. (They will correspond to the kinetic energy space, potential energy space, and two spaces of boundary conditions.) We assume that $\mathbb V$ is continuously embedded into $\mathbb H$. The abstract differential operator is a bounded linear operator
$\mathsf A_\star: \mathbb V \to \mathbb H.$ Some of the boundary conditions are encoded in a bounded linear and surjective operator $\mathsf B:\mathbb V\to\mathbb M_2.$
We assume the following property:
\begin{equation}\label{eq:0}
C_1^\star \| U\|_{\mathbb V}
	\le \| U\|_{\mathbb H}+ \| \mathsf A_\star U \|_{\mathbb H}
	\le C_2^\star \| U\|_{\mathbb V}
	\qquad \forall U\in \mathbb V.
\end{equation}
The rightmost inequality is a consequence of the boundedness of $\mathsf A_\star$ and of the injection of $\mathbb V$ into $\mathbb H$. 
We next define the operator
\[
\mathsf A:=\mathsf A_\star|_{D(\mathsf A)}:
	D(\mathsf A)\subset\mathbb H\to \mathbb H, \qquad D(\mathsf A):=\mathrm{Ker}\,\mathsf B.
\]
This operator will be treated as an unbounded operator. We assume that $\pm \mathsf A$ are maximal dissipative, i.e.,
\begin{equation}\label{eq:1}
(\mathsf A U,U)_{\mathbb H}=0 \qquad \forall U\in D(\mathsf A)
\end{equation}
and
\begin{equation}\label{eq:2}
\mathsf I \pm \mathsf A : D(\mathsf A)\to \mathbb H
	\quad\mbox{are surjective}.
\end{equation}
The maximal dissipativity of $-\mathsf A$ guarantees time-reversibility, but will not be used for the estimates.
Neither of $\pm\mathsf A_\star$ can be dissipative in their domain $\mathbb V$, since otherwise they would be dissipative extensions of a maximal dissipative operator. As a consequence of the above hypotheses $\mathsf A$ is the infinitesimal generator of a $C_0$-group of isometries in $\mathbb H$. 
(This is part of the Lumer-Philips Theorem cf. \cite[Theorem 4.5.1]{Kesavan:1989}.)
In particular $D(\mathsf A)$ is dense in $\mathbb H$ and, therefore, so is $\mathbb V$. Another bounded linear operator
$
\mathsf G : \mathbb M_1 \to \mathbb H
$
 deals with some `natural' boundary conditions that are added as source terms. A final hypothesis is the following: given arbritrary $\Xi:=(\xi,\chi)\in \mathbb M:=\mathbb M_1\times \mathbb M_2$, there exists a unique solution to
\begin{equation}\label{eq:3}
U\in \mathbb V, \qquad
U=\mathsf A_\star U+\mathsf G \xi, \qquad \mathsf B U=\chi,
\end{equation}
and
\begin{equation}\label{eq:4}
\| U\|_{\mathbb H}+\| U\|_{\mathbb V} \le C_{\mathrm{lift}} \|\Xi\|_{\mathbb M}.
\end{equation}
The operator $\mathsf L: \mathbb M\to \mathbb V$ given by the solution of \eqref{eq:3} will be referred to as a lifting operator.

\paragraph{The problem.} Given data functions
$F:[0,\infty)\to \mathbb H$ and $\Xi=(\xi,\chi):[0,\infty)\to \mathbb M$, 
we look for $U:[0,\infty) \to \mathbb V$ such that
\begin{subequations}\label{eq:5}
\begin{align}
\dot U(t) & = \mathsf A_\star U(t)+\mathsf G \xi(t)+F(t) & & t\ge 0,\\
\mathsf BU(t) &=\chi(t) & & t\ge 0,\\
U(0) & =0.
\end{align}
\end{subequations}
One might wonder why we keep the term $\mathsf G \xi$ separated from the `source terms' in $F$. The reason is that we expect $\| \mathsf G\|$ to be difficult to control and we will deal with this term through the lifting operator $\mathsf L$. In the end, the price to pay will be the need for higher regularity in time for $\xi$ than $F$, even if they apparently play similar roles in the equation. Note that if $U$ is continuous as a $\mathbb V$-valued function, then, necessarily $\chi(0)=0$. (The term related to $\mathsf G$ will not be used in this paper, but it is added here since this slightly more extended theory is used in other works \cite{HaSa:Sub}.)

\paragraph{The main results.} We will deal with the spaces
\[
W^k(X):=\{ f\in \mathcal C^{k-1}([0,\infty);X)\,:\, 
	f^{(k)}\in L^1((0,\infty);X),\,
	f^{(\ell)}(0)=0,\quad  \, 0\le \ell\le k-1\}.
\]
The space $W^k(X)$ can be characterized as the set of functions $f:[0,\infty)\to X$ such that $Ef\in W^k_+(\mathbb R;X).$

\begin{theorem}\label{the:3.1}
If $F_0\in W^1(\mathbb H)$ and $\Xi:=(\xi,\chi)\in W^2(\mathbb M)$, then equation \eqref{eq:5} has a unique solution $U\in \mathcal C^1([0,\infty);\mathbb H)\cap \mathcal C([0,\infty);\mathbb V)$ and for all $t\ge 0$:
\begin{subequations}\label{eq:8}
\begin{align}
		\label{eq:8a}
\| U(t)\|_{\mathbb H} \le
	 & C_{\mathrm{lift}} \left( \int_0^t \| \Xi(\tau)\|_{\mathbb M}\mathrm d\tau
							+ 2\int_0^t  \| \dot\Xi(\tau)\|_{\mathbb M} 	\mathrm d\tau\right)
		+ \int_0^t \| F(\tau)\|_{\mathbb H}\mathrm d\tau,\\	
		\label{eq:8b}
\| \dot U(t)\|_{\mathbb H} \le
	& C_{\mathrm{lift}} \left( \int_0^t \| \dot\Xi(\tau)\|_{\mathbb M}\mathrm d\tau
						+ 2\int_0^t  \| \ddot\Xi(\tau)\|_{\mathbb M}\mathrm d\tau\right)
		+ \int_0^t \| \dot F(\tau)\|_{\mathbb H}\mathrm d\tau.
\end{align}
\end{subequations}
\end{theorem}

\begin{proof}
Let $U_{\mathrm{NH}}:=\mathsf L\Xi\in W^2(\mathbb V)$. and let $U_0:[0,\infty)\to D(\mathsf A)$ be the unique solution of
\begin{equation}\label{eq:7}
\dot U_0(t)=\mathsf A U_0(t)+ F_0(t) \quad t\ge 0, \qquad U_0(0)=0,
\end{equation}
where $F_0:=F+U_{\mathrm{NH}}-\dot U_{\mathrm{NH}}=F+\mathsf L(\Xi-\dot\Xi)\in W^1(\mathbb H)$. By 
\cite[Corollary 2.5]{Pazy:1983}, there exists a unique solution of \eqref{eq:7} $U_0\in \mathcal C^1([0,\infty);\mathbb H)\cap \mathcal C([0,\infty);D(\mathsf A))$. Moreover
\[
\| U_0(t)\|_{\mathbb H}\le \int_0^t \| F_0(\tau)\|_{\mathbb H}\mathrm d\tau,
	\qquad
\| \dot U_0(t)\|_{\mathbb H}\le \int_0^t \| \dot F_0(\tau)\|_{\mathbb H}\mathrm d\tau
	\qquad \forall t\ge 0.
\]
Adding \eqref{eq:3} and \eqref{eq:7}, it is clear that $U:=U_{\mathrm{NH}}+U_0$ is a solution of \eqref{eq:5} and $U\in \mathcal C^1([0,\infty);\mathbb H)\cap \mathcal C([0,\infty);\mathbb V))$. Using \eqref{eq:4}, we can bound
\begin{align*}
\| U(t)\|_{\mathbb H} \le
	& C_{\mathrm{lift}} \left( \| \Xi(t)\|_{\mathbb M} 
		+ \int_0^t \| \Xi(\tau)-\dot\Xi(\tau)\|_{\mathbb M}\mathrm d\tau\right)
		+ \int_0^t \| F(\tau)\|_{\mathbb H}\mathrm d\tau \\
	\le & C_{\mathrm{lift}} \left( \int_0^t \| \Xi(\tau)\|_{\mathbb M}\mathrm d\tau
							+ 2\int_0^t  \| \dot\Xi(\tau)\|_{\mathbb M}\mathrm d\tau\right)
		+ \int_0^t \| F(\tau)\|_{\mathbb H}\mathrm d\tau.			
\end{align*}
We can prove \eqref{eq:8b} similarly.  Uniqueness of the solution to \eqref{eq:5} follows from uniqueness of the solution of
\[
\dot V(t)=\mathsf AV(t) \quad t\ge 0, \qquad V(0)=0
\]
and the fact that $D(\mathsf A)=\mathrm{Ker}\,\mathsf B$. 
\end{proof}

Note that by \eqref{eq:0} and \eqref{eq:8},
\begin{align}\label{boundV}
\nonumber
C_1^\star \| U(t)\|_{\mathbb V}
\le & \| U(t)\|_{\mathbb H}+\| \mathsf A_\star U(t)\|_{\mathbb H} \\
\nonumber
\le & \| U(t)\|_{\mathbb H}+\|\dot U(t)\|_{\mathbb H}+\| F(t)\|_{\mathbb H}+\|\mathsf G\xi(t)\|_{\mathbb H}\\
\le & C_{\mathrm{lift}} \left( \int_0^t \| \Xi(\tau)\|_{\mathbb M}\mathrm d\tau
		+ 3 \int_0^t \| \dot\Xi(\tau)\|_{\mathbb M}\mathrm d\tau
		+ 2 \int_0^t \| \ddot\Xi(\tau)\|_{\mathbb M}\mathrm d\tau\right) \\
\nonumber
	& + \int_0^t\| F(\tau)\|_{\mathbb H}\mathrm d\tau + 2\int_0^t\| \dot F(\tau)\|_{\mathbb H}\mathrm d\tau
	+ \| \mathsf G\|\,\| \Xi(t)\|_{\mathbb M}.
\end{align}

\begin{theorem}[Distributional extension]\label{the:3.2}
Let $U$ be the solution of \eqref{eq:5} for data in the hypotheses of Theorem \ref{the:3.1}, and let $\underline U:=EU$, $\underline \xi:=E\xi$, $\underline\chi:=E\chi$, $\underline F=E F$. Then $\underline U$ is the unique solution of 
\begin{equation}\label{eq:3.9}
\underline U\in \TD(\mathbb V), \qquad
\dot{\underline U}=\mathsf A_\star\underline U+ \mathsf G\underline\xi+\underline F,
\qquad
\mathsf B\underline U=\underline\chi.
\end{equation}
\end{theorem}

\begin{proof}
Let
\[
C_\Xi:=\int_0^\infty \| \ddot \Xi(\tau)\|_{\mathbb M}\mathrm d\tau,
\qquad
C_F:=\int_0^\infty \| \dot F(\tau)\|_{\mathbb H}\mathrm d\tau.
\]
The bound \eqref{eq:8a} implies that
\[
\| U(t)\|_{\mathbb H} \le C_{\mathrm{lift}}C_\Xi (1+2t) + C_F(1+t),
\]
and by \eqref{boundV}
\[
\| U(t)\|_{\mathbb V} \le C_{\mathrm{lift}}C_\Xi (1+3t+t^2) + C_F(1+t)+\| \mathsf G\| C_\Xi t.
\]
This implies that $U$ is polynomially bounded for large $t$ as an $\mathbb H$- and $\mathbb V$-valued function. Therefore $\underline U:=E U\in \TD(\mathbb V)$ and $\underline U\in \TD(\mathbb H)$. As seen in Section \ref{sec:2}, since $U\in \mathcal C^1([0,\infty);\mathbb H)$ and $U(0)=0$, then
\[
\tfrac{\mathrm d}{\mathrm dt} \underline U= E\dot U
\]
as $\mathbb H$-valued distributions. Since $E$ is a linear operator that commutes with any operator that is independent of the time variable, then \eqref{eq:3.9} is satisfied.
\end{proof}

\section{The general result}\label{sec:4}

We are next going to define a particular (while quite general in purpose) example of dynamical system as those studied in Section \ref{sec:3}. We take $\mathbb H:=L^2(\RdG)\times \mathbf L^2(\RdG)$, $\mathbb V:=H^1(\RdG)\times \Hdiv{\RdG}$ and $\mathsf A_\star U=\mathsf A_\star(u,\mathbf v):=(\nabla\cdot\mathbf v,\nabla u)$. We now consider two closed spaces
\[
X_h\subset \Hmhalf, \qquad Y_h\subset\Hhalf,
\]
and their polar sets
\begin{align*}
X_h^\circ & :=\{ \varphi\in \Hhalf\,:\, \langle \mu^h,\varphi\rangle_\Gamma=0 \quad\forall \mu^h\in X_h\},\\
Y_h^\circ &:=\{ \eta\in \Hmhalf\,:\,\langle\eta,\varphi^h\rangle_\Gamma=0 \quad\forall \varphi^h\in Y_h\}.
\end{align*}
We next consider the spaces with homogeneous abstract transmission conditions
\begin{align*}
U_h &:=\{ u\in H^1(\RdG)\,:\, \gamma^+ u\in X_h^\circ, \quad\jump{\gamma u}\in Y_h\},\\
\Vh &:=\{\bff v \in \Hdiv{\RdG}\,:\, \jump{\ntr\bff v}\in X_h, \quad\ntr^-\bff v\in Y_h^\circ\},
\end{align*}
as well as the operator $\mathsf A:D(\mathsf A)\subset\mathbb H\to\mathbb H$, where $D(\mathsf A):=U_h\times \Vh$. 

\paragraph{One remark.} We can fit this example in the framework of Section \ref{sec:3} by using the operator $\mathsf B(u,\bff v):=(\gamma^+u|_{X_h}, \ntr^-\bff v|_{Y_h}, \jump{\gamma u}|_{Y_h^\circ}, \jump{\ntr\bff v}|_{Y_h^\circ})$ taking values in $\mathbb M_2:=X_h^*\times (X_h^\circ)^*\times Y_h^*\times (Y_h^\circ)^*$. Let us clarify this point. The trace $\gamma^+ u$ is in $\Hhalf=\Hmhalf^*$ and we can therefore understand $\gamma^+ u|_{X_h}:X_h \to \R$ as an element of the dual space of $X_h$, which we denote $X_h^*$, defined by $X_h \ni \mu^h \mapsto \langle \mu^h,\gamma^+ u\rangle_\Gamma$.  The same explanation works for the three remaining components of $\mathsf B$. Note that $D(\mathsf A)=\mathrm{Ker}\,\mathsf B$: for instance, $\gamma^+u|_{X_h}=0$ is the same as $\gamma^+ u\in X_h^\circ$ and $\jump{\gamma u}|_{Y_h^\circ}=0$ is equivalent to $\jump{\gamma u}\in (Y_h^\circ)^\circ=Y_h$, because $Y_h$ is closed.

\paragraph{A second remark.} If we choose the conditions based on $\gamma^- u$ and $\ntr^+ \bff v$ we obtain a very similar problem for which everything we will prove still holds. This second particular problem contains some additional examples as concrete instances, but all the results that this new problem would provide can be proved by adequately choosing the right-hand sides in the problem we will study. 

\begin{proposition}[Infinitesimal generator]\label{prop:4.1}
 The operators $\pm \mathsf A:D(\mathsf A)\subset \mathbb H\to \mathbb H$ are maximal dissipative.
\end{proposition}

\begin{proof}
Note first that for all $(u,\bff v)\in U_h\times \Vh$,
\begin{align*}
(\mathsf A(u,\bff v),(u,\bff v))_{\mathbb H}
	& = (\nabla\cdot\bff v,u)_\RdG+(\bff v,\nabla u)_\RdG \\
	& =\langle\ntr^- \bff v,\gamma^- u\rangle_\Gamma-\langle \ntr^+\bff v,\gamma^+ u\rangle_\Gamma \\
	& = \langle \ntr^-\bff v,\jump{\gamma u}\rangle_\Gamma
		+\langle \jump{\ntr \bff v},\gamma^+ u\rangle_\Gamma =0,
\end{align*}
which proves that $\pm \mathsf A$ are dissipative. Let now $(f,\bff g)\in \mathbb H$. We look for
$(u,\bff v)\in U_h\times \Vh$ satisfying
\[
u\pm \nabla\cdot\bff v= f, \qquad \bff v\pm \nabla u=\bff g,
\]
with both equations taking place in $\RdG$. To do that we first solve the coercive variational problem
\begin{equation}\label{eq:4.1}
u\in U_h \qquad
	(u,w)_\RdG+(\nabla u,\nabla w)_\RdG =(f,w)_\RdG\pm (\bff g,\nabla w)_\RdG \quad \forall w\in U_h,
\end{equation}
and then define $\bff v:=\mp \nabla u+\bff g \in L^2(\RdG)$. We then substitute $\nabla u=\mp (\bff v-\bff g)$ in \eqref{eq:4.1}  and simplify to obtain
\begin{equation}\label{eq:4.2}
(u,w)_\RdG\mp (\bff v,\nabla w)_\RdG=(f,w)_\RdG \qquad \forall w\in U_h.
\end{equation}
Testing \eqref{eq:4.2} with a general $\mathcal C^\infty$ function with compact support in $\RdG$, it follows that $u\pm \nabla\cdot\bff v=f$ and therefore $\bff v\in \Hdiv{\RdG}$. Note that we only need to prove the transmission conditions related to $\bff v$ to finish the proof (of surjectivity of $\mathsf I\pm \mathsf A$). We now substitute $f=u\pm \nabla\cdot\bff v$ in \eqref{eq:4.1} to prove that
\[
(\bff v,\nabla w)_\RdG+(\nabla\cdot\bff v,w)_\RdG=0 \qquad \forall w\in U_h,
\]
or equivalently, 
\begin{equation}\label{eq:4.3}
\langle \ntr^-\bff v,\jump{\gamma w}\rangle_\Gamma
	+\langle\jump{\ntr\bff v},\gamma^+ w\rangle_\Gamma = 0 \qquad \forall w\in U_h.
\end{equation}
Since the operator $H^1(\RdG)\ni w\longmapsto (\jump{\gamma w},\gamma^+ w)\in \Hhalf\times \Hhalf$ is surjective, it is easy to see that so is $U_h \ni w \longmapsto (\jump{\gamma w},\gamma^+ w)\in Y_h\times X_h^\circ$. Therefore \eqref{eq:4.3} implies that $\ntr^-\bff v\in Y_h^\circ$ and $\jump{\ntr \bff v}\in X_h$, which finishes the proof.
\end{proof}

For convenience, we introduce the space $\mathbb M(\Gamma):=\Hhalf\times \Hhalf\times \Hmhalf\times\Hmhalf$, endowed with the product norm, denoted $\|\cdot\|_{\pm1/2,\Gamma}$.

\begin{proposition}[Lifting operator]\label{prop:4.2}
For all $(\rho_1,\rho_2,\psi_1,\psi_2)\in \mathbb M(\Gamma)$ there exists a unique $(u,\bff v)\in \mathbb V$ satisfying
\begin{subequations}\label{eq:4.4}
\begin{align}
& u =\nabla\cdot\bff v, & & \bff v=\nabla u, \\
& \gamma^+ u-\rho_1 \in X_h^\circ & & \jump{\gamma u}-\rho_2\in Y_h, \\
& \ntr^-\bff v-\psi_1\in Y_h^\circ & & \jump{\ntr\bff v}-\psi_2\in X_h.
\end{align}
\end{subequations}
The solution of \eqref{eq:4.4} can be bounded as follows
\begin{equation}
\| u\|_{1,\RdG}=\|\bff v\|_{\divv,\RdG}\le C_\Gamma \| (\rho_1,\rho_2,\psi_1,\psi_2)\|_{\pm 1/2,\Gamma},
\end{equation}
where $C_\Gamma$ only depends on the geometry of the problem through constants related to the trace operator
and its optimal right-inverse.
\end{proposition}

\begin{proof}
Solving problem \eqref{eq:4.4} is equivalent to solving
\begin{subequations}\label{eq:4.5}
\begin{align}
& -\Delta u+u=0, \\
& \gamma^+ u-\rho_1 \in X_h^\circ & & \jump{\gamma u}-\rho_2\in Y_h, \\
& \nd^-u-\psi_1\in Y_h^\circ & & \jump{\nd u}-\psi_2\in X_h.
\end{align}
\end{subequations}
and then defining $\bff v=\nabla u$. However, \eqref{eq:4.5} is equivalent to
\begin{subequations}\label{eq:4.6}
\begin{align}
& u\in H^1(\RdG), \\
& \gamma^+u-\rho_1\in X_h^\circ, \qquad  \jump{\gamma u}-\rho_2\in Y_h,\\
& (u,w)_\RdG+(\nabla u,\nabla w)_\RdG
=\langle \psi_1,\jump{\gamma w}\rangle_\Gamma
	+\langle\psi_2,\gamma^+ w\rangle_\Gamma \qquad \forall w\in U_h.
\end{align}
\end{subequations}
Problem \eqref{eq:4.6} is a coercive variational problem in $U_h$ after decomposing the solution in the form $u=u_d+u_0$, where $\gamma^+u_d=\rho_1$, $\jump{\gamma u_d}=\rho_2$, and $u_0\in U_h$. Note that to build $u_d$ we just need to invert the trace conditions $\gamma^+ u_d=\rho_1$ and $\gamma^-u_d=\rho_1+\rho_2$, which can be done independently of the spaces $X_h$ and $Y_h$. Note also that the coercivity and boundedness constant of the bilinear and linear forms in \eqref{eq:4.6} are independent of these spaces as well.
\end{proof}

Propositions \ref{prop:4.1} and \ref{prop:4.2} have verified the conditions on the operator and boundary conditions given in Section \ref{sec:3}. We are now ready to use Theorems \ref{the:3.1} and \ref{the:3.2} to derive results on a wave equation associated to the operators $(\mathsf A,\mathsf B)$. Since we work with the second order wave equations (given in Section \ref{sec:2}), the problem will be translated to a first order (in space and time) system in the proof of the next result.

\begin{theorem}\label{the:4.3}
Let $(\alpha_1,\alpha_2)\in W^2_+(\R;\Hhalf^2)$ and $(\beta_1,\beta_2)\in W^1_+(\R;\Hmhalf^2)$. The unique solution of
\begin{subequations}\label{eq:4.8}
\begin{align}
u\in \TD(H^1_\Delta(\RdG)) \quad &  \ddot u=\Delta u,\\
&\gamma^+u-\alpha_1\in X_h^\circ,  && \jump{\gamma u}-\alpha_2 \in Y_h,\\
& \nd^-u-\beta_1 \in Y_h^\circ,  && \jump{\nd u}-\beta_2\in X_h,
\end{align}
\end{subequations}
satisfies
\begin{equation}
u\in \mathcal C^1_+(\R;L^2(\RdG))\cap \mathcal C^0_+(\R;H^1(\RdG))
\end{equation}
and
\begin{equation}\label{eq:4.10}
\| u(t)\|_{1,\RdG}\le 3C_\Gamma 
\sum_{\ell=0}^2 \int_0^t \| (\alpha_1^{(\ell)},\alpha_2^{(\ell)},	
							\beta_1^{(\ell-1)},\beta_2^{(\ell-1)})(\tau)\|_{\pm1/2,\Gamma}
							\mathrm d\tau	\qquad\forall t\ge 0,
\end{equation}
where $\beta^{(-1)}:=\partial^{-1}\beta$. If $(\alpha_1,\alpha_2)\in W^3_+(\R;\Hhalf^2)$ and $(\beta_1,\beta_2)\in W^2_+(\R;\Hmhalf^2)$, then
\begin{equation}\label{eq:4.11}
\| \nabla u(t)\|_{\divv,\RdG}\le 3C_\Gamma 
\sum_{\ell=1}^3 \int_0^t \| (\alpha_1^{(\ell)},\alpha_2^{(\ell)},	
							\beta_1^{(\ell-1)},\beta_2^{(\ell-1)})(\tau)\|_{\pm1/2,\Gamma}
							\mathrm d\tau	\qquad\forall t\ge 0.
\end{equation}
The constant $C_\Gamma$ in \eqref{eq:4.10} and \eqref{eq:4.11} is the one of Proposition \ref{prop:4.2} and is, therefore, independent of the choice of $X_h$ and $Y_h$.
\end{theorem}

\begin{proof}
If $u$ is the solution of \eqref{eq:4.8}, then $(u,\bff v):=(u,\partial^{-1}\nabla u)$ is the solution to
\begin{subequations}
\begin{align}
(u,\bff v)\in \TD(\mathbb V) \quad
	& \dot u=\nabla\cdot\bff v, && \dot{\bff v}=\nabla u,\\
	&\gamma^+u-\alpha_1\in X_h^\circ,  && \jump{\gamma u}-\alpha_2 \in Y_h,\\
	& \ntr^-\bff v-\partial^{-1}\beta_1 \in Y_h^\circ,  && \jump{\ntr \bff v}-\partial^{-1}\beta_2\in X_h.
\end{align}
\end{subequations}
Note that $(\alpha_1,\alpha_2,\partial^{-1}\beta_1,\partial^{-1}\beta_2)|_{(0,\infty)}\in W^2(\mathbb M(\Gamma))$. We can then apply Theorems \ref{the:3.1} and \ref{the:3.2} noticing that $\| u(t)\|_{1,\RdG}=\|(u,\dot{\bff v})(t)\|_{\mathbb H}$, which means that we need part of the bounds \eqref{eq:8a} and \eqref{eq:8b} to prove \eqref{eq:4.10}.  The bound \eqref{eq:4.11}, requiring additional data regularity follows from the following observations: (a) the operator $(\alpha_1,\alpha_2,\beta_1,\beta_2)\mapsto u$ is a convolution operator and, therefore, commutes with time differentiation; (b) $\| \nabla u(t)\|_{\divv,\RdG}=
\| (\ddot u,\dot{\bff v})(t)\|_{\mathbb H}$.  This means that we can use the bounds \eqref{eq:8} for data $\dot\Xi$ to obtain the estimate \eqref{eq:4.11}.
\end{proof}

As explained in the proof of Theorem \ref{the:4.3}, the operator $(\alpha_1,\alpha_2,\beta_1,\beta_2)\mapsto u$ is a convolution operator in the sense of operator-and-vector-valued distributions. Therefore, it commutes with differentiation and we can apply a shifting argument to show that it defines a bounded map
\[
W^k_+(\R;\Hhalf^2)\times W^{k-1}_+(\R;\Hmhalf^2)
\to \mathcal C^{k-1}_+(\R;L^2(\RdG))\cap \mathcal C^{k-2}_+(\R;H^1(\RdG))
\]
for all $k\ge 2$. Note also that \eqref{eq:4.10} can be used directly to provide bounds for the quantities
\[
\| \gamma^\pm u(t)\|_{1/2,\Gamma}, \quad \| \jump{\gamma u}(t)\|_{1/2,\Gamma},
	\quad \| \ave{\gamma u}(t)\|_{1/2,\Gamma},
\]
while \eqref{eq:4.11} can be invoked to bound
\[
\| \nd^\pm u(t)\|_{-1/2,\Gamma}, \quad \| \jump{\nd u}(t)\|_{-1/2,\Gamma},
	\quad \| \ave{\nd u}(t)\|_{-1/2,\Gamma}.
\]
In both cases, additional constants that depend only on the geometry (through bounds for the trace operator and its optimal right-inverse) will be introduced, the key point here being that all constants are independent of the choice of $X_h$ and $Y_h$.

Before we move to the next step of this paper (examining particular cases of Theorem \ref{the:4.3}) let us state a simple but relevant result that follows from a straightforward uniqueness argument.

\begin{proposition}\label{prop:4.4}
Let $\Pi_h^X:\Hmhalf\to X_h$ and $\Pi_h^Y:\Hhalf \to Y_h$ be the best approximation operators onto $X_h$ and $Y_h$ respectively. The solution of problem \eqref{eq:4.8}	with data $(\alpha_1,\alpha_2,\beta_1,\beta_2)$ is the same as the solution with data $(\alpha_1,\alpha_2-\Pi_h^Y \alpha_2,\beta_1,\beta_2-\Pi_h^X \beta_2)$. Therefore, the bounds \eqref{eq:4.10} and \eqref{eq:4.11} still hold if we substitute $\alpha_2$ by $\alpha_2-\Pi_h^Y \alpha_2$ and $\beta_2$ by  $\beta_2-\Pi_h^X \beta_2$.
\end{proposition}	

\paragraph{Another remark.}
Note that in the context of our abstract framework of Section \ref{sec:3} the transmission-boundary conditions in \eqref{eq:4.4} can be written as 	$\mathsf B(u,\bff v)=\xi$, where $\xi:=(\rho_1|_{X_h}, \psi_1|_{Y_h},\rho_2|_{Y_h^\circ},\psi_2|_{X_h^\circ})$ and $\| \xi\|_{\mathbb M_2}\le \| (\rho_1,\rho_2,\psi_1,\psi_2)\|_{\pm1/2}$.

\section{The examples}\label{sec:5}

This section examines different choices of $X_h$ and $Y_h$, as well as of the data functions $(\alpha_1,\alpha_2,\beta_1,\beta_2)$ in Theorem \ref{the:4.3}, to describe: retarded potentials, boundary integral operators, time domain integral equations for scattering problems,  Galerkin semidiscretizations of the latter, etc. Once we have identified these problems, we will be able to provide estimates using the general theory of Section \ref{sec:4}. We want to emphasize that some of these results had already been proved in the literature. In all cases we get improvements with respect to Laplace domain estimates. In some cases we get improvements (especially when we refer to still non-optimized approaches in \cite{DoSa:2013}, \cite{BaLaSa:2015}) or just the same estimates proved in a much simpler way (the second order in time and space analysis of \cite{Sayas:2014} requires much more additional work in the reconciliation of the estimates for a strong form of the dynamical system and its associated distributional version). Finally, we show that some `clever' choices of $X_h$ and $Y_h$ provide estimates for the forward operators, a detail that had been missed in \cite{Sayas:2014} and the papers that led to that monograph.

For ease of reference we next write the interior and exterior Dirichlet and Neumann problems for the wave equation
\begin{subequations}
\begin{align}\label{eq:5.1a}
u\in \TD(H^1_\Delta(\RdG)) \quad &
\ddot u = \Delta u \qquad  \gamma^\pm  u= \alpha^\pm \\
\label{eq:5.1b}
u\in \TD(H^1_\Delta(\RdG)) \quad &
\ddot u=\Delta u \qquad   \nd^\pm u=\beta^\pm.
\end{align}
\end{subequations}
From this moment on $c_\Gamma$ is a generic constant independent of the choice of the spaces $X_h$ and $Y_h$. It typically includes the influence of the constant $C_\Gamma$ of Proposition \ref{prop:4.2} and of the trace operators $\gamma^\pm: H^1(\RdG)\to \Hhalf$, $\ntr^\pm:\Hdiv{\RdG}\to \Hmhalf$. The exterior Dirichlet-to-Neumann map is the operator $\alpha^+\to\nd^+u$, where $u$ solves \eqref{eq:5.1a} (the value of $\alpha^-$ is not relevant). Definitions for the interior DtN and exterior-interior NtD operators follow likewise. To shorten notation we will write, for instance, $(\frac12+\CK)*\beta:=\tfrac12\beta+\CK*\beta$. Properly speaking, the scalar factor is multiplying $\delta_0\otimes \mathcal I$, where $\mathcal I$ is the associated identity operator (in this case in $\Hhalf$), $\delta_0$ is the scalar time-domain Dirac delta distribution and $\otimes$ denotes the tensor product of a scalar distribution with an operator that does not depend on time.

In all the coming bounds we will use the cummulative seminorm
\[
H_2(f,t;X):=\sum_{\ell=0}^2 \int_0^t \| f^{(\ell)}(\tau)\|_X\mathrm d\tau.
\]

\subsection{Continuous operators}\label{sec:51}

\paragraph{Potentials and integral operators.}
If we choose $X_h=\{0\}$ and $Y_h=\{0\}$ and data $(\alpha_1,\alpha_2,\beta_1,\beta_2)=(\times, \varphi,\times,\lambda)$ (the components $\alpha_1$ and $\beta_1$ of the data set are ignored by void transmission conditions in \eqref{eq:4.8}, which we denote by writing the $\times$ symbol), then the solution of \eqref{eq:4.8} is
$u=\CS*\lambda-\CD*\varphi$
and
\[
\ave{\gamma u}=\CV*\lambda-\CK*\varphi, 
\qquad
\ave{\nd u}=\CK^t*\lambda+\CW*\varphi.
\]
Estimates for the single layer potential and associated integral operators follow from Theorem \ref{the:4.3}:
\begin{alignat}{6}\label{eq:cts1}
\nonumber
\lambda\in W^1_+(\R;\Hmhalf), \qquad
	& \CS *\lambda \in \mathcal C(\R;H^1(\Rd)), \\
\nonumber
	& \| (\CS*\lambda)(t)\|_{1,\Rd}  \le\, c_\Gamma H_2(\partial^{-1}\lambda,t;\Hmhalf), \\
	& \CV *\lambda \in \mathcal C(\R;\Hhalf),\\
\nonumber
	& \| (\CV*\lambda)(t)\|_{1/2,\Gamma}  \le\, c_\Gamma H_2(\partial^{-1}\lambda,t;\Hmhalf),& \\
\nonumber
\lambda\in W^2_+(\R;\Hmhalf),\qquad 
	& \CK^t *\lambda \in \mathcal C(\R;\Hmhalf), \\
\nonumber
	& \| (\CK^t*\lambda)(t)\|_{-1/2,\Gamma}  \le\, c_\Gamma H_2(\lambda,t;\Hmhalf).
\end{alignat}
Similar results can be found for the double layer potential and associated integral operators:
\begin{alignat}{6}\label{eq:cts2}
\nonumber
\varphi\in W^2_+(\R;\Hhalf),\qquad
	& \CD *\varphi \in \mathcal C(\R;H^1(\RdG)), \\
\nonumber
	& \| (\CD*\varphi)(t)\|_{1,\RdG}  \le\, c_\Gamma H_2(\varphi,t;\Hhalf), \\
	& \CK *\varphi \in \mathcal C(\R;\Hhalf),\\
\nonumber
	& \| (\CK*\varphi)(t)\|_{1/2,\Gamma}  \le\, c_\Gamma H_2(\varphi,t;\Hhalf),& \\
\nonumber
\varphi\in W^3_+(\R;\Hhalf),\qquad 
	& \CW *\varphi \in \mathcal C(\R;\Hmhalf),\\
\nonumber
	& \| (\CW *\varphi)(t)\|_{-1/2,\Gamma}  \le\, c_\Gamma H_2(\dot\varphi,t;\Hhalf).
\end{alignat}

\paragraph{Integral formulations for Dirichlet problems.}
 Let $X_h=\Hmhalf$ and $Y_h=\{0\}$. The solution for data $(\alpha_1,\alpha_2,\times,\times)$ is
\[
u=\CS*\CV^{-1}*\alpha_1+ (\CS*\CV^{-1}*(\tfrac12 + \CK)-\CD)*\alpha_2.
\]
Note that
\[
\jump{\nd u}=\CV^{-1}*\alpha_1 + \CV^{-1}*(\tfrac12+\CK)*\alpha_2
\]
and $\gamma^+ u=\alpha_1$, $\gamma^- u=\alpha_1+\alpha_2$. The data $(\alpha,0,\times,\times)$ correspond to a single layer representation $u=\CS*\lambda$ of the solution to the Dirichlet problem \eqref{eq:5.1a} with $\alpha^\pm=\alpha$ and $\CV*\lambda=\alpha$. The data $(0,\alpha,\times,\times)$ correspond to a direct representation of the solution of \eqref{eq:5.1a} with $\alpha^+=0$, $\alpha^-=\alpha$:
\[
u=\CS*\lambda-\CD*\alpha, \qquad \CV*\lambda=(\tfrac12+\CK)*\alpha, 
\qquad \lambda=\nd^- u.
\]
In particular we have an estimate of the interior Dirichlet-to-Neumann map. We next collect some estimates for both problems.
\begin{alignat}{6}\label{eq:cts3}
\alpha\in W^2_+(\R;\Hhalf),\qquad
\nonumber
	& \CS*\CV^{-1}*\alpha \in \mathcal C(\R;H^1(\R^d)), \\
\nonumber
	& \| (\CS*\CV^{-1}*\alpha)(t)\|_{1,\R^d} \le c_\Gamma H_2(\alpha,t;\Hhalf),\\
	& u:=(\CS*\CV^{-1}*(\tfrac12 + \CK)-\CD)*\alpha \in \mathcal C(\R;H^1(\Omega_-)), \\
\nonumber
	& \| u(t)\|_{1,\Omega_-} \le c_\Gamma H_2(\alpha,t;\Hhalf),\\
\nonumber
\alpha\in W^3_+(\R;\Hhalf),\qquad
	&\CV^{-1}*\alpha \in \mathcal C(\R;\Hmhalf), \\
\nonumber
	& \| (\CV^{-1}*\alpha)(t)\|_{-1/2,\Gamma}\le c_\Gamma H_2(\dot\alpha,t;\Hhalf),\\
	&\lambda:=\mathrm{DtN}^-(\alpha)=\CV^{-1}*(\tfrac12+\CK)*\alpha \in \mathcal C(\R;\Hmhalf),\\
\nonumber
	&\|\lambda(t)\|_{-1/2,\Gamma}\le c_\Gamma H_2(\dot\alpha,t;\Hhalf).
\end{alignat}
If we solve \eqref{eq:4.8} with the given choice of spaces and data $(\alpha,-\alpha,\times,\times)$ we solve the Dirichlet problem \eqref{eq:5.1a} with data $\alpha^+=\alpha$ and $\alpha^-=0$. Therefore, we also have an estimate for the exterior DtN operator.

\paragraph{Improvements.} The estimates in \eqref{eq:cts1} and \eqref{eq:cts2} improve on the estimates in \cite{DoSa:2013} by removing the dependence on time of some of the constants in the energy estimates.   In addition to the sharper bounds found here, these results do not require the detailed cut-off process that was necessary for that analysis.  Estimates for the Dirichlet problem were previously derived in \cite{BaLaSa:2015}.  The present analysis improves on these bounds in a number of ways: the derivation is simpler, the bounds \eqref{eq:cts3} are sharper, and we require less regularity of the data in the time variable to prove our results.

\paragraph{Integral formulations for Neumann problems.}
Take now $X_h=\{0\}$ and $Y_h=\Hhalf$. The solution of \eqref{eq:4.8} with data $(\times,\times,\beta_1,\beta_2)$ is
\[
u=-\CD*\CW^{-1}*\beta_1+(\CS+\CD*\CW^{-1}*(\tfrac12+\CK^t))*\beta_2
\]
and, therefore,	
\[
\jump{\gamma u}=\CW^{-1}*\beta_1+\CW^{-1}*(\tfrac12+\CK^t)*\beta_2,
\]
while $\nd^- u=\beta_1$, $\nd^+ u=\beta_1-\beta_2$. The particular case $(\times,\times,\beta,0)$ corresponds to a double layer representation $u=-\CD*\varphi$ of the solution of the Neumann problem \eqref{eq:5.1b} with data $\beta^\pm=\beta$, and $\varphi$ computed as the solution of $\CW*\varphi=\beta$. The data $(\times,\times,0,-\beta)$ provide a direct representation of the solution of the exterior Neumann problem with vanishing interior data ($\beta^+=\beta$, $\beta^-=0$):
\[
u=-\CS*\beta+\CD*\varphi, \qquad 
\CW*\varphi=-(\tfrac12+\CK^t)*\beta,\qquad
\varphi=\gamma^+ u.
\]
Here are some associated estimates:
\begin{alignat*}{6}
\beta\in W^2_+(\R;\Hhalf) \qquad
	& \CD*\CW^{-1}*\beta\in \mathcal C(\R;H^1(\RdG)),\\
	& \| (\CD*\CW^{-1}*\beta)(t)\|_{1,\RdG}\le c_\Gamma H_2(\beta,t;\Hhalf),\\
	& \CW^{-1}*\beta\in \mathcal C(\R;\Hhalf),\\
	& \|  (\CW^{-1}*\beta)(t)\|_{1/2,\Gamma} \le c_\Gamma H_2(\beta,t;\Hhalf),\\
	& u=(\CS+\CD*\CW^{-1}*(\tfrac12+\CK^t))*\beta)\in \mathcal C(\R;H^1(\Omega_+)),\\
	& \| u(t)\|_{1,\Omega_+} \le  c_\Gamma H_2(\beta,t;\Hhalf),\\
	& \varphi:=\mathrm{NtD}^+(\beta) =\CW^{-1}*(\tfrac12+\CK^t)*\beta,\\
	& \|\varphi(t)\|_{1/2,\Gamma}\le c_\Gamma H_2(\beta,t;\Hhalf).
\end{alignat*}
For a direct formulation of the interior Neumann problem we use $(\times,\times,\beta,\beta)$. 

\subsection{Semidiscretization of integral equations}

In this section we derive results about semidiscretization in space of the equations in Sections \ref{sec:51}. From this moment on, we will not spell out the regularity requirements on the problem data. They will be assumed to be such that the right-hand side of the bounds is finite.

\paragraph{Equations for the Dirichlet problem.}
Let $X_h$ be finite dimensional and $Y_h=\{0\}$. The corresponding transmission conditions are
\begin{equation}\label{eq:5.2}
\gamma^+ u-\alpha_1\in X_h^\circ, \qquad
	\jump{\gamma u}=\alpha_2, \qquad
	\jump{\nd u}-\beta_2 \in X_h,
\end{equation}
with an additional void equation associated to the other boundary data: $\nd^- u-\beta_1\in \Hmhalf$. The data $(\alpha,0,\times,0)$ correspond to solving the semidiscrete equations:
\begin{equation}\label{eq:5.3}
\lambda^h \in X_h, \qquad \CV*\lambda^h-\alpha\in X_h^\circ, \qquad u^h:=\CS*\lambda^h.
\end{equation}
We can bound
\begin{equation}\label{eq:5.4}
\| u^h(t)\|_{1,\RdG}\le c_\Gamma H_2(\alpha,t;\Hhalf),
	\qquad
	\|\lambda^h(t)\|_{-1/2,\Gamma} \le c_\Gamma H_2(\dot\alpha,t;\Hhalf).
\end{equation}
This is a Galerkin semidiscretization of
\begin{equation}\label{eq:5.4bis}
\CV*\lambda=\alpha, \qquad u=\CS*\lambda.
\end{equation}
The data $(0,\alpha,\times,0)$ correspond to
\begin{equation}\label{eq:5.5}
\lambda^h\in X_h, \qquad \CV*\lambda^h-(\tfrac12+\CK)*\alpha\in X_h^\circ,
	\qquad u^h=\CS*\lambda^h -\CD*\alpha,
\end{equation}
and yields bounds identical to \eqref{eq:5.4}. This is a Galerkin semidiscretization of
\begin{equation}\label{eq:5.5bis}
\CV*\lambda=(\tfrac12+\CK)*\alpha, \qquad u=\CS*\lambda-\CD*\alpha.
\end{equation}
Data $(0,0,\times,\lambda)$ produces a semidiscretization-in-space bound for both \eqref{eq:5.3} and \eqref{eq:5.5}. 
Let $\wt u$ is the solution of \eqref{eq:4.8} with this choice of space and data and let $\lambda^h:=\lambda-\jump{\nd \wt u}$, then 
\[
\lambda^h\in X_h, \qquad
	\CV*(\lambda^h-\lambda)\in X_h^\circ, \qquad
 	\wt u=\CS*(\lambda-\lambda^h).
\]
We have two scenarios covered. In the first one, we are approximating \eqref{eq:5.4bis} by \eqref{eq:5.3}. In the second one, we are approximating  \eqref{eq:5.5bis} by \eqref{eq:5.5}. In both cases $\wt u=u-u^h$ and we can estimate (recall Proposition \ref{prop:4.4})
\begin{subequations}\label{eq:5.6}
\begin{align}
\| u(t)-u^h(t)\|_{1,\RdG} & \le c_\Gamma H_2(\partial^{-1}\lambda-\Pi_h^X\partial^{-1}\lambda,t;\Hmhalf),\\
\|\lambda(t)-\lambda^h(t)\|_{-1/2,\Gamma}& \le c_\Gamma H_2(\lambda-\Pi_h^X\lambda,t;\Hmhalf).
\end{align}
\end{subequations}
The bounds \eqref{eq:5.4} are {\em stability} estimates for Galerkin semidiscretization of two different equations associated to the convolution operator $\lambda\mapsto\CV*\lambda$, while \eqref{eq:5.6} are {\em error estimates} for those semidiscretizations.

\paragraph{Equations for the Neumann problem.}
Let $X_h=\{0\}$ and $Y_h$ be finite dimensional. The associated non-void transmission conditions are
\[
\jump{\gamma u}-\alpha_2\in Y_h, \qquad \nd^- u-\beta_1\in Y_h^\circ, \qquad \jump{\nd u}=\beta_2.
\]
With data $(\times,0,\beta,0)$ we are solving
\begin{equation}\label{eq:5.9}
\varphi^h\in Y_h,\qquad
	\CW*\varphi^h-\beta\in Y_h^\circ, \qquad u=-\CD*\varphi^h,
\end{equation}
as an approximation of the indirect formulation of the interior-exterior Neumann problem (see Section \ref{sec:51})
\begin{equation}\label{eq:5.10}
\CW*\varphi=\beta, \qquad u=-\CD*\varphi.
\end{equation}
With data $(\times,0,0,\beta)$ we are instead solving
\begin{equation}\label{eq:5.11}
\varphi^h\in Y_h, \qquad \CW*\varphi^h-(\tfrac12+\CK^t)*\beta\in Y_h^\circ,
	\qquad u=\CS*\beta+\CD*\varphi^h
\end{equation}
as an approximation to the direct formulation of the exterior Dirichlet problem ($\varphi=\gamma^+ u$)
\begin{equation}\label{eq:5.12}
\CW*\varphi=(\tfrac12+\CK^t)*\beta,\qquad
	u=\CS*\beta+\CD*\varphi.
\end{equation}
In both cases, we derive stability estimates
\[
\| u^h(t)\|_{1,\RdG}+\|\varphi^h(t)\|_{1/2,\Gamma}\le c_\Gamma H_2(\partial^{-1}\beta,t;\Hmhalf).
\]
If the solution with data $(\times,\varphi,0,0)$ is denoted $\wt u$ and $\varphi^h:=\jump{\gamma\wt u}+\varphi$, then $\wt u=\CD*(\varphi-\varphi^h)$ and we are proving error estimates for the approximation of \eqref{eq:5.10}  by \eqref{eq:5.9} and of \eqref{eq:5.12} by \eqref{eq:5.11}:
\[
\| u(t)-u^h(t)\|_{1,\RdG}+\|\varphi(t)-\varphi^h(t)\|_{1/2,\Gamma}
	\le c_\Gamma H_2(\varphi-\Pi_h^Y\varphi,t;\Hhalf).
\]

\subsection{Symmetric Galerkin solvers}\label{sec:5.3}

In this section we outline what kind of problems we solve when we take discrete spaces $X_h$ and $Y_h$ or, in the limit, $X_h=\Hmhalf$ and $Y_h=\Hhalf$. 
With data $(\alpha^+,0,\beta^-,0)$, we have $\varphi^h:=\jump{\gamma u}\in Y_h$, $\lambda^h:=\jump{\nd u}\in X_h$, and we can represent $u=\CS*\lambda^h-\CD*\varphi^h$. Therefore
\[
\gamma^+ u=\CV*\lambda^h-(\tfrac12+\CK)*\varphi^h,
\qquad
\nd^- u=(\tfrac12+\CK^t)*\lambda^h+\CW*\varphi^h. 
\]
We will give an interpretation of what $(u,\varphi^h,\lambda^h)$ is later on. At this stage we can state the stability estimates
\begin{align*}
\| \varphi^h(t)\|_{1/2,\Gamma}
+ \| u(t)\|_{1,\RdG} & \le c_\Gamma \big( H_2(\alpha^+,t;\Hhalf)+H_2(\partial^{-1}\beta^-,t;\Hmhalf)\big),\\
\| \lambda^h(t)\|_{-1/2,\Gamma}
	&\le c_\Gamma \big( H_2(\dot\alpha^+,t;\Hhalf)+H_2(\beta^-,t;\Hmhalf)\big).
\end{align*}

\paragraph{Symmetric formulation for Dirichlet problem.}
The data $(\alpha^+,0,0,0)$ provide the semidiscrete system
\[
\CV*\lambda^h-(\tfrac12+\CK)*\varphi^h-\alpha^+\in X_h^\circ,
\qquad
(\tfrac12+\CK^t)*\lambda^h+\CW*\varphi^h \in Y_h^\circ,
\]
which is the $X_h\times Y_h$ Galerkin semidiscretization of the symmetric formulation
\begin{equation}\label{eq:5.13}
\left[\begin{array}{cc} \CV  & -\tfrac12-\CK \\ \tfrac12+\CK^t & \CW \end{array}\right]
*\left[\begin{array}{c} \lambda \\ \varphi \end{array}\right]=
\left[\begin{array}{c} \alpha^+ \\ 0 \end{array}\right].
\end{equation}
The system \eqref{eq:5.13} is a realization of the symmetric form for the exterior Dirichlet-to-Neumann (Steklov-Poincar\'e) operator
\begin{equation}\label{eq:5.14}
\big(\CV+(\tfrac12+\CK)*\CW^{-1}*(\tfrac12+\CK^t)\big)*\lambda=\alpha^+,
\end{equation}
via the introduction of the artificial variable $\varphi=-\CW^{-1}*(\tfrac12+\CK^t)*\lambda$, which, in the continuous case, is a copy of $-\alpha^+$. The exact system \eqref{eq:5.13} is recovered when $X_h=\Hmhalf$ and $Y_h=\Hhalf$. When $X_h$ is finite dimensional and $Y_h=\Hhalf$, we obtain a non-practicable method consisting of using an $X_h$ Galerkin semidiscretization of \eqref{eq:5.14}. 

\paragraph{Symmetric formulation for Neumann problem.} The data $(0,0,\beta^-,0)$ correspond to a semidiscretization of
\begin{equation}\label{eq:5.15}
\left[\begin{array}{cc} \CV  & -\tfrac12-\CK \\ \tfrac12+\CK^t & \CW \end{array}\right]
*\left[\begin{array}{c} \lambda \\ \varphi \end{array}\right]=
\left[\begin{array}{c} 0 \\ \beta^- \end{array}\right],
\end{equation}
which, inverting the first equation can be reduced to
\begin{equation}
\big( \CW+(\tfrac12+\CK^t)*\CV^{-1}*(\tfrac12+\CK)\big)*\varphi=\beta^-.
\end{equation}
This is the Steklov-Poincar\'e formula for the Neumann-to-Dirichlet operator. 

\paragraph{Associated error operators.} Let $u$ be the solution to \eqref{eq:4.8} with data $(\alpha^+,0,\beta^-,0)$ for $X_h=\Hmhalf$ and $Y_h=\Hhalf$. Let $\varphi:=\jump{\gamma u}$ and $\lambda:=\jump{\nd u}$. We now solve again \eqref{eq:4.8} with the same data but changing the spaces $X_h$ and $Y_h$ to be finite dimensional: we denote its solution by $u^h$ and define $\varphi^h:=\jump{\gamma u^h}$, $\lambda^h:=\jump{\nd u^h}$. The errors between exact and semidiscrete solutions can be studied by applying Theorem \ref{the:4.3} (and Proposition \ref{prop:4.4}) to data $(0,\varphi,0,\lambda)$ with the discrete spaces $X_h$ and $Y_h$. We then have bounds for the errors
\begin{align*}
& \| u(t)-u^h(t)\|_{1,\RdG}
+\| \varphi(t)-\varphi^h(t)\|_{1/2,\Gamma} \\
& \hspace{1cm} \le c_\Gamma \big( H_2(\varphi-\Pi_h^Y\varphi,t;\Hhalf)+H_2(\partial^{-1}\lambda-\Pi_h^X\partial^{-1}\lambda,t;\Hmhalf)\big),
\end{align*}
and
\[
 \| \lambda(t)-\lambda^h(t)\|_{-1/2,\Gamma} \le c_\Gamma \big( H_2(\dot\varphi-\Pi_h^Y\dot\varphi,t;\Hhalf)+H_2(\lambda-\Pi_h^X\lambda,t;\Hmhalf)\big).
\]

\subsection{Further Applications}

\paragraph{Mixed boundary conditions.}

Let $\Gamma$ be divided into two relatively open sets $\Gamma_D$ and $\Gamma_N$ such that $\Gamma_D\cap \Gamma_N=\emptyset$ and $\overline\Gamma_D\cup\overline\Gamma_N=\Gamma$. We consider the space
\[
H^{1/2}(\Gamma_D):=\{ \varphi|_{\Gamma_D}\,:\, \varphi\in H^{1/2}(\Gamma)\}.
\]
This space can be endowed with the image norm of the restriction operator $R:H^{1/2}(\Gamma)\to H^{1/2}(\Gamma_D)$ or with any other equivalent norm. We define
\[
\wt H^{1/2}(\Gamma_N):=\mathrm{Ker}\,R=\{ \varphi\in H^{1/2}(\Gamma)\,:\, \varphi|_{\Gamma_D}=0\}. 
\]
Since $R$ is bounded and surjective, the adjoint operator $R^*:(H^{1/2}(\Gamma_D))^* \to H^{-1/2}(\Gamma)$ is injective and has closed range. We then define
\[
\wt H^{-1/2}(\Gamma_D):=\mathrm{Range}\,R^*=(\mathrm{Ker}\,R)^\circ=\wt H^{1/2}(\Gamma_D)^\circ.
\]
This set is isomorphic to $H^{1/2}(\Gamma_D)^*$. Formally speaking, elements of $\wt H^{-1/2}(\Gamma_D)$ vanish on $\Gamma_N$. Consider now that we have an extension of the Dirichlet data and of the Neumann data, so that we have at our disposal elements $(\alpha,\beta)\in \TD(\Hhalf\times \Hmhalf)$. We consider an exterior solution of the wave equation, extended by zero to the interior domain, we can then write the mixed boundary conditions as
\[
\gamma^+ u-\alpha \in \wt H^{1/2}(\Gamma_N), \qquad \nd^+ u-\beta\in \wt H^{-1/2}(\Gamma_D),
\]
or, taking into account the vanishing value of $u$ in the interior domain
\begin{subequations}
\begin{equation}
\jump{\gamma u}+\alpha \in \wt H^{1/2}(\Gamma_N), 
\qquad
\jump{\nd u}+\beta\in \wt H^{-1/2}(\Gamma_D).
\end{equation}
Defining $\varphi:=\gamma^+ u-\alpha\in \wt H^{1/2}(\Gamma_N)$ and $\lambda:=\nd^+ u-\beta\in \wt H^{-1/2}(\Gamma_D)$, we can represent the solution using Kirchhoff's formula
\[
u=-\CS*(\beta+\lambda)+\CD*(\alpha+\varphi)
\]
and note that 
\[
\left[\begin{array}{c}-\gamma^+ u+\alpha \\ -\nd^- u \end{array}\right]
=\left[\begin{array}{cc} \CV & -\tfrac12-\CK \\ \tfrac12+\CK^t & \CW \end{array}\right]
*\left[\begin{array}{c} \beta+\lambda  \\ \alpha+\varphi\end{array}\right] 
+\left[\begin{array}{c} \alpha \\ 0 \end{array}\right]
=\left[\begin{array}{c} -\varphi \\ 0 \end{array}\right],
\]
which implies
\begin{equation}
\gamma^+ u-\alpha \in \wt H^{1/2}(\Gamma_N)=\wt H^{-1/2}(\Gamma_D)^\circ,
\qquad
\nd^- u \in \wt H^{1/2}(\Gamma_N)^\circ.
\end{equation}
\end{subequations}
This means that the choice of spaces $X_h=\wt H^{-1/2}(\Gamma_D)$ and $Y_h=\wt H^{1/2}(\Gamma_N)$ allows us to recover the transmission conditions of problem \eqref{eq:4.8} with data $(\alpha,-\alpha,0,-\beta)$. Note that in this case the spaces $X_h$ and $Y_h$ are related by $X_h^\circ=Y_h$ and $Y_h^\circ=X_h$.

If we take finite dimensional subspaces $X_h\subset \wt H^{-1/2}(\Gamma_D)$ and $Y_h\subset\wt H^{1/2}(\Gamma_N)$ the theory covers for the semidiscrete Galerkin scheme
\[
(\lambda^h,\varphi^h)\in X_h\times Y_h,
\qquad
\left[\begin{array}{cc} \CV & -\tfrac12-\CK \\ \tfrac12+\CK^t & \CW \end{array}\right]
*\left[\begin{array}{c} \beta+\lambda^h  \\ \alpha+\varphi^h\end{array}\right]
+\left[\begin{array}{c} \alpha \\ 0 \end{array}\right]
\in X_h^\circ\times Y_h^\circ,
\]
followed by the potential reconstruction
\[
u^h:=-\CS*(\beta+\lambda^h)+\CD*(\alpha+\varphi^h).
\]
The stability and semidiscretization error estimates of Section \ref{sec:5.3} still hold.

\paragraph{Dirichlet and Neumann screens.} We can understand a screen $\Gscr$ as any geometric set in $\R^d$ that can be completed to a closed boundary $\Gamma$ of a Lipschitz domain $\Omega_-$. Let us go back to the notation of the above paragraph with $\Gamma_D=\Gscr$ ($\Gamma_N$ is the part we have added to $\Gscr$ to create $\Gamma$). If we take $X_h=\wt H^{-1/2}(\Gscr)$ and $Y_h=\{0\}$, the data $(0,0,\times,\lambda)$ correspond to studying the single layer potential $u=\CS*\lambda$ for $\lambda\in \wt H^{-1/2}(\Gscr)$ and its trace $\CV*\lambda$. Properly speaking, the kind of bounds we obtain for $u$ are given in $H^1(\RdG)$. However, since $\jump{\gamma u}=0$ and $\jump{\nd u}\in \wt H^{-1/2}(\Gscr)$, these bounds are automatically extended to $H^1(\R^d\setminus\Gscr)$. Similarly, we can understand that the actual single layer operator on the screen is defined by $R(\CV*\lambda)\in \TD(H^{1/2}(\Gscr))$, so that it is valued on a space of functions defined only on $\Gscr$. 

The data $(\alpha,0,\times,0)$ correspond to solving the Dirichlet problem on the screen using a single layer potential representation
\begin{equation}\label{eq:5.18}
\lambda\in \TD(\wt H^{-1/2}(\Gscr)) \qquad R(\CV*\lambda-\alpha)=0, 
\qquad u=\CS*\lambda.
\end{equation}
The choice of a finite dimensional space $X_h$ provides Galerkin semidiscretizations of \eqref{eq:5.18} and the associated semidiscretization error analysis. We emphasize that the general bounds given in Section \ref{sec:5.3} cover all these new situations.

The Neumann case can be studied by letting $X_h=\{0\}$ and $Y_h$ be either $\wt H^{1/2}(\Gscr)$ or its finite dimensional subspace. With this choice of spaces we are dealing with a screen on which we define a double layer potential and two-sided Neumann boundary conditions can be imposed.

\section{The final words}

Let us finally point out some simple extensions and applications of the techniques developed in this paper:
\begin{itemize}
\item The joint treatment of many problems (forward operators, solution operators, semidiscrete solution operators, screen problems) can be also used in the Laplace domain analysis, thus collecting many existing results as particular choices of spaces in a general transmission problem (`parameterized' in the two spaces $X_h$ and $Y_h$).
\item The application of these techniques to BEM-FEM coupled modeling of scattering by non-homogeneous obstacles is being explored in \cite{HaSa:Sub}.
\item The results on scattering by penetrable homogeneous obstacles proved in \cite{QiSa:2014} can be reproved with the techniques of this paper. The techniques are equivalent to those used in \cite{QiSa:2014} and no improvement in the bounds is obtained. The proofs with this first order equation approach are simpler though.
\item All the results in this paper can be extended verbatim to the elastic wave equation.
\item The application of these ideas to Maxwell equations actually precedes this paper \cite{QiSa:2015}, given the fact that the second order equation ideas \cite{Sayas:2013d, Sayas:2014} seem not to apply to the functional space setting of the layer potentials for electromagnetism.
\end{itemize}
The global transmission problem of Section \ref{sec:4} makes an effort in collecting all problems under one roof (a {\em one ring to rule them all} of sorts), emphasizing that the analytical tools of many apparently different situations follow a clear pattern. In a way, we hope this paper will guide and simplify future endeavors in the analysis of time domain integral equations. 

\bibliographystyle{abbrv}
\bibliography{biblio}

\end{document}